\documentclass[12pt]{article}

\usepackage{amsmath,amsthm,amssymb}
\usepackage[hidelinks]{hyperref}

\newtheorem{theorem}{Theorem}
\newtheorem{proposition}[theorem]{Proposition}
\newtheorem{lemma}[theorem]{Lemma}

\newcommand{\hsum}[1]{\widehat{#1}}
\newcommand{\seqnum}[1]{\href{https://oeis.org/#1}{\rm \underline{#1}}}

\title{The greedy 3-sumfree sequence $S_{1,g,g+1}$}
\author{Orion Shtrezi}
\date{}

\begin{document}

\maketitle

\begin{abstract}
For every integer $g\ge 2$, we determine exactly which integers occur in the greedy 3-sumfree sequence that starts with $1$, $g$, and $g+1$. This gives a direct proof of a conjecture of Bosma, Bruin, Fokkink, Grube, Reuijl, and Tromp. We also give an explicit description of the eventual periodicity, including both the preperiod and the repeating block.
\end{abstract}

\section{Introduction}\label{sec:intro}

Let $x<y<z$ be positive integers. The greedy 3-sumfree sequence $S_{x,y,z}$ is the increasing sequence that begins with $x,y,z$ and then continues by the rule that each new term is the least positive integer exceeding the preceding term that is not a sum of three distinct earlier terms.

Bosma et al.\ \cite[Theorem~15 and Conjecture~16]{BosmaEtAl} proved explicit descriptions of $S_{1,2,3}$, $S_{1,3,4}$, and $S_{1,4,5}$ with the automatic theorem prover Walnut, and they conjectured a uniform formula for $S_{1,g,g+1}$ for every $g\ge 2$, which they verified mechanically for $2\le g\le 10$. The cases $S_{1,2,3}$ and $S_{1,3,4}$ appear in the OEIS \cite{OEIS} as \seqnum{A026471} and \seqnum{A026475}, respectively; a note in entry \seqnum{A026471} records an earlier proof of the case $g=2$ by Matthew Akeran. In this note we prove the full conjecture by a direct interval-sumset argument.

For integers $u$ and $v$, we write
\[
[u,v]=\{n\in\mathbb Z : u\le n\le v\},
\]
with the convention that $[u,v]=\varnothing$ when $u>v$.

\begin{theorem}\label{thm:main}
Let $g\ge 2$ be an integer. Then the greedy 3-sumfree sequence $S_{1,g,g+1}$ is characterized by
\begin{equation}\label{eq:main}
\begin{aligned}
 z\in S_{1,g,g+1}
 &\iff
 z\in\{1,2g+1,6g+1\} \\
 &\hspace{2.5em}\text{or } z\bmod(10g+3)\in[g,2g]\cup[6g+2,7g+1].
\end{aligned}
\end{equation}
In particular, after the first $g+4$ terms, the sequence modulo $10g+3$ is periodic with period length $2g+1$.
\end{theorem}

Section~\ref{sec:prelim} fixes notation and reduces Theorem~\ref{thm:main} to two propositions, which are proved in Sections~\ref{sec:omitted} and~\ref{sec:included}; the induction that completes the proof is carried out at the end of Section~\ref{sec:included}.

\section{Notation and reduction}\label{sec:prelim}

For subsets $X,Y\subseteq\mathbb Z$, we write
\[
X+Y=\{x+y : x\in X,\ y\in Y\}.
\]
For a set $X\subseteq\mathbb Z$, we define
\[
\begin{aligned}
\hsum{2}X &= \{x+y : x,y\in X,\ x<y\}, \\
\hsum{3}X &= \{x+y+z : x,y,z\in X,\ x<y<z\}.
\end{aligned}
\]
We use the elementary identities
\begin{equation}\label{eq:hat-identities}
\hsum{2}[u,v]=[2u+1,2v-1],
\qquad
\hsum{3}[u,v]=[3u+3,3v-3],
\end{equation}
whenever the intervals on the left are nonempty.

Set
\[
 m=10g+3,
 \qquad
 e=2g+1,
 \qquad
 f=6g+1,
\]
and define
\[
 I_k=km+[g,2g],
 \qquad
 J_k=km+[6g+2,7g+1]
 \qquad
 (k\ge 0).
\]
Then the right-hand side of \eqref{eq:main} is the set
\[
 A_g=\{1,e,f\}\cup\bigcup_{k\ge 0}(I_k\cup J_k).
\]
Since $g\ge 2$, the residues of $1$, $e$, and $f$ modulo $m$ lie outside $[g,2g]\cup[6g+2,7g+1]$, so the sets $\{1\}$, $\{e\}$, $\{f\}$, $I_0,I_1,\ldots$, and $J_0,J_1,\ldots$ are pairwise disjoint and partition $A_g$.

We call an integer \emph{forced} if it is a sum of three distinct elements of $A_g$. Since all elements of $A_g$ are positive, any representation
\[
 n=a+b+c
\]
with distinct $a,b,c\in A_g$ automatically satisfies $a,b,c<n$. Thus Theorem~\ref{thm:main} follows once we prove the following two propositions.

\begin{proposition}\label{prop:omitted}
Every integer $n>g+1$ with $n\notin A_g$ is forced.
\end{proposition}

\begin{proposition}\label{prop:included}
No element of $A_g$ is forced.
\end{proposition}

\section{Representing the omitted integers}\label{sec:omitted}

\begin{proof}[Proof of Proposition~\ref{prop:omitted}]
For the initial block, let
\[
B=[g,2g+1],
\qquad
D=[6g+1,7g+1],
\]
so that $B=I_0\cup\{e\}$ and $D=\{f\}\cup J_0$ are subsets of $A_g$.
The omitted integers above $g+1$ in block $0$ are
\[
[2g+2,6g]\cup[7g+2,m].
\]
By \eqref{eq:hat-identities} we have
\[
1+\hsum{2}B=[2g+2,4g+2],
\qquad
\hsum{3}B=[3g+3,6g],
\]
so these two intervals cover $[2g+2,6g]$. Likewise,
\[
1+B+D=[7g+2,9g+3],
\qquad
(7g+1)+\hsum{2}B=[9g+2,11g+2].
\]
Since $m=10g+3\le 11g+2$, these two intervals cover $[7g+2,m]$.

Now let $k\ge 1$. The omitted integers in block $k$ are exactly
\[
km+[1,g-1],
\qquad
km+[2g+1,6g+1],
\qquad
km+[7g+2,m].
\]
For the first gap we have
\[
((k-1)m+7g+1)+e+[g+2,2g]=km+[1,g-1].
\]
For the middle gap we have
\[
1+B+I_k=km+[2g+1,4g+2],
\qquad
(km+2g)+\hsum{2}B=km+[4g+1,6g+1].
\]
For the last gap we have
\[
1+D+I_k=km+[7g+2,9g+2],
\qquad
(km+7g+1)+\hsum{2}B=km+[9g+2,11g+2].
\]
Again $m=10g+3\le 11g+2$, so these intervals cover all of $km+[7g+2,m]$.

Each displayed identity gives a representation by three distinct elements of $A_g$. The operators $\hsum{2}$ and $\hsum{3}$ enforce distinctness inside $B$, and the remaining singled-out summand lies outside the relevant copy of $B$. In the mixed sums $1+B+D$, $1+B+I_k$, and $1+D+I_k$, the three summands come from pairwise disjoint sets, and the same holds for the representation of the first gap, whose summands come from $J_{k-1}$, $\{e\}$, and $[g+2,2g]\subseteq I_0$. Therefore every omitted integer $n>g+1$ is forced.
\end{proof}

\section{Excluding the admitted integers}\label{sec:included}

Define the periodic complement-pattern set
\[
C=
\bigcup_{k\ge 0}\bigl(km+[0,g-1]\bigr)
\cup
\bigcup_{k\ge 0}\bigl(km+[2g+1,6g+1]\bigr)
\cup
\bigcup_{k\ge 0}\bigl(km+[7g+2,m]\bigr).
\]
By construction,
\begin{equation}\label{eq:A-cap-C}
A_g\cap C=\{1,e,f\}.
\end{equation}

\begin{lemma}\label{lem:sums-in-C}
Every sum of three distinct elements of $A_g$ lies in $C$.
\end{lemma}

\begin{proof}
We examine each feasible block-membership pattern for three distinct elements of $A_g$. By the partition of $A_g$ noted in Section~\ref{sec:prelim}, the patterns below are exhaustive, and summands taken from different parts of the partition are automatically distinct. Inside the interval blocks $I_k$ and $J_k$ we ignore distinctness, since doing so can only enlarge the corresponding sumsets.

For the three exceptional terms we have
\[
1+e+f=8g+3\in C.
\]

For two exceptional terms and one periodic term, we obtain
\begin{align*}
1+e+I_p &= pm+[3g+2,4g+2],
& 1+e+J_p &= pm+[8g+4,9g+3], \\
1+f+I_p &= pm+[7g+2,8g+2],
& 1+f+J_p &= (p+1)m+[2g+1,3g], \\
e+f+I_p &= pm+[9g+2,10g+2],
& e+f+J_p &= (p+1)m+[4g+1,5g].
\end{align*}

For one exceptional term and two periodic terms, we obtain
\begin{align*}
1+I_p+I_q &= (p+q)m+[2g+1,4g+1], \\
1+I_p+J_q &= (p+q)m+[7g+3,9g+2], \\
1+J_p+J_q &= (p+q+1)m+[2g+2,4g], \\
e+I_p+I_q &= (p+q)m+[4g+1,6g+1], \\
e+I_p+J_q &= (p+q)m+[9g+3,m]\cup(p+q+1)m+[1,g-1], \\
e+J_p+J_q &= (p+q+1)m+[4g+2,6g], \\
f+I_p+I_q &= (p+q)m+[8g+1,10g+1], \\
f+I_p+J_q &= (p+q+1)m+[3g,5g-1], \\
f+J_p+J_q &= (p+q+1)m+[8g+2,10g].
\end{align*}

Finally, for three periodic terms, let $K=p+q+r$. Then
\begin{align*}
I_p+I_q+I_r &= Km+[3g,6g], \\
I_p+I_q+J_r &= Km+[8g+2,m]\cup(K+1)m+[1,g-2], \\
I_p+J_q+J_r &= (K+1)m+[3g+1,6g-1], \\
J_p+J_q+J_r &= (K+1)m+[8g+3,\min\{11g,m\}]\cup(K+2)m+[1,g-3].
\end{align*}
Each right-hand side is contained in $C$. Therefore every sum of three distinct elements of $A_g$ lies in $C$.
\end{proof}

\begin{proof}[Proof of Proposition~\ref{prop:included}]
By Lemma~\ref{lem:sums-in-C} and \eqref{eq:A-cap-C}, the only elements of $A_g$ that could be forced are $1$, $e$, and $f$.

The smallest sum of three distinct elements of $A_g$ is
\[
1+g+(g+1)=2g+2,
\]
so neither $1$ nor $e=2g+1$ can be forced.

It remains to exclude $f=6g+1$. The elements of $A_g$ below $f$ are exactly
\[
1,\ g,\ g+1,\ \ldots,\ 2g,\ e.
\]
The largest sum of three distinct such elements is
\[
e+2g+(2g-1)=(2g+1)+2g+(2g-1)=6g<f.
\]
Therefore $f$ is not forced either.
\end{proof}

\begin{proof}[Proof of Theorem~\ref{thm:main}]
We prove by induction on $n$ that
\[
n\in S_{1,g,g+1}
\iff
n\in A_g.
\]
The claim is clear for $1\le n\le g+1$, since both sets contain exactly the integers $1$, $g$, and $g+1$ in this range.

Now let $n>g+1$, and assume that the claim holds for every positive integer smaller than $n$. If $n\notin A_g$, then Proposition~\ref{prop:omitted} gives a representation
\[
n=a+b+c
\]
with distinct $a,b,c\in A_g$. Since $a,b,c<n$, the induction hypothesis gives $a,b,c\in S_{1,g,g+1}$. Hence the greedy rule excludes $n$.

If $n\in A_g$, then Proposition~\ref{prop:included} shows that $n$ is not a sum of three distinct elements of $A_g$. By the induction hypothesis, the earlier terms of $S_{1,g,g+1}$ are exactly the elements of $A_g$ below $n$. Therefore $n$ is not a sum of three distinct earlier terms of $S_{1,g,g+1}$, so the greedy rule admits $n$.

This completes the induction, and therefore $S_{1,g,g+1}=A_g$. The description \eqref{eq:main} follows immediately.

After the first $g+4$ terms
\[
1,\ g,\ g+1,\ \ldots,\ 2g,\ 2g+1,\ 6g+1,
\]
the sequence of residues repeats as
\[
6g+2,\ 6g+3,\ \ldots,\ 7g+1,\ g,\ g+1,\ \ldots,\ 2g
\]
modulo $m=10g+3$. This repeating block has length $2g+1$.
\end{proof}

\bigskip

\noindent
2020 \emph{Mathematics Subject Classification}: Primary 11B83; Secondary 11B75.

\smallskip

\noindent
\emph{Keywords:} greedy sequence, sumfree sequence, integer sequence, eventual periodicity.

\bigskip

\noindent
(Concerned with sequences \seqnum{A026471} and \seqnum{A026475}.)


\begin{thebibliography}{9}

\bibitem{BosmaEtAl}
W.~Bosma, R.~Bruin, R.~Fokkink, J.~Grube, A.~Reuijl, and T.~Tromp,
Using Walnut to solve problems from the OEIS,
\emph{J. Integer Seq.} \textbf{28} (2025), Article 25.3.8.

\bibitem{OEIS}
N.~J.~A. Sloane et al., \emph{The On-Line Encyclopedia of Integer Sequences}, 2026. Available at \url{https://oeis.org}.

\end{thebibliography}
\end{document}